\documentclass[12pt]{amsart}

\usepackage{amsfonts,amsmath,latexsym,amssymb,verbatim,amsbsy,amsthm}

\usepackage{soul}

\usepackage[top=1in, bottom=1in, left=1in, right=1in]{geometry}

\usepackage[dvipsnames]{xcolor}

\usepackage[colorlinks=true, pdfstartview=FitV, linkcolor=RoyalBlue,citecolor=ForestGreen, urlcolor=blue]{hyperref}

\numberwithin{equation}{section}

\renewcommand{\geq}{\geqslant}
\renewcommand{\ge}{\geqslant}
\renewcommand{\leq}{\leqslant}
\renewcommand{\le}{\leqslant}

\theoremstyle{plain}
\newtheorem{THEOREM}{Theorem}[section]

\newtheorem{theorem}[THEOREM]{Theorem}

\newtheorem{lemma}[THEOREM]{Lemma}

\theoremstyle{definition}

\theoremstyle{remark}

\newcommand{\thm}[1]{Theorem~\ref{#1}}
\newcommand{\lem}[1]{Lemma~\ref{#1}}

\def \a {\alpha} 

\def \d {\delta}

\def \g {\gamma}

\def \e {\varepsilon}

\def \L {\Lambda}
\def \n {\nabla}

\def \r {\rho}

\def \cL {\mathcal{L}}

\def \cM {\mathcal{M}}

\def \cP {\mathcal{P}}

\def \cF {\mathcal{F}}

\def \aL {{\mathcal L}}
\def \aT {{\mathcal T}}

\newcommand{\Z}{\ensuremath{\mathbb{Z}}}   
\newcommand{\R}{\ensuremath{\mathbb{R}}}   
\newcommand{\T}{\ensuremath{\mathbb{T}}}   

\def \p {\partial}
\def \ra {\rightarrow}

\def \dx  {\, \mbox{d}x}
\def \dy  {\, \mbox{d}y}
\def \dz  {\, \mbox{d}z}

\def \dt  {\, \mbox{d}t}

\def \dr {\, \mbox{d}r}

\def \ddt  {\frac{\mbox{d\,\,}}{\mbox{d}t}}
\def \DDt  {\frac{\mbox{D\,\,}}{\mbox{D}t}}
\def \dD  {\mbox{D}}

\def \aL {\mathcal{L}}

\begin{document}

\title[Eulerian dynamics with a commutator forcing: fractional diffusion]{Eulerian dynamics with  a commutator forcing III.\\ Fractional diffusion of order $0<\alpha<1$}

\author{Roman Shvydkoy}
\address{Department of Mathematics, Statistics, and Computer Science, M/C 249,\\
University of Illinois, Chicago, IL 60607, USA}
\email{shvydkoy@uic.edu}

\author{Eitan Tadmor}
\address{Department of Mathematics, Center for Scientific Computation and Mathematical Modeling (CSCAMM), and Institute for Physical Sciences \& Technology (IPST), University of Maryland, College Park\newline
Current address: Institute for Theoretical Studies (ITS),
ETH-Zurich, Clausiusstrasse 47, CH-8092 Zurich, Switzerland}
\email{tadmor@cscamm.umd.edu}

\date{\today}



\thanks{\textbf{Acknowledgment.} Research was supported in part by NSF grants DMS16-13911, RNMS11-07444 (KI-Net) and ONR grant N00014-1512094 (ET) and by NSF grant DMS 1515705 and the College of LAS, University of Illinois at Chicago (RS). 
Both authors thank the Institute for Theoretical Studies (ITS) at ETH-Zurich for the hospitality.}

\dedicatory{To Edriss Titi with friendship and admiration}

\begin{abstract} We continue our study of hydrodynamic models of self-organized evolution of agents with singular interaction kernel $\phi(x) = |x|^{-(1+\alpha)}$. Following our works \cite{ST2017a,ST2017b} which focused on the range $1\leq \alpha <2$, and Do et. al. \cite{DKRT2017} which covered the range $0<\alpha<1$, in this paper we revisit the latter case and give a short(-er) proof of global in time existence of smooth solutions, together with a full description of their long time dynamics. Specifically, we prove that starting from any initial condition in $(\rho_0,u_0) \in H^{2+\alpha}\times H^3$, the solution approaches exponentially fast to a flocking state solution consisting of a  wave $\bar{\rho}=\rho_\infty(x-t\bar{u}))$ traveling with a  constant velocity determined by the conserved average velocity $\bar{u}$.  The convergence is accompanied by exponential decay of all higher order derivatives of $u$.
\end{abstract}

\maketitle

\setcounter{tocdepth}{2}
\tableofcontents

\section{Introduction and statement of main results.}
We  continue our study of one-dimensional Eulerian dynamics driven by forcing with  a commutator structure initiated in \cite{ST2017a,ST2017b}:
\begin{equation}\label{eq:FH}
\left\{
\begin{split}
\rho_t +(\rho u)_x&=0,\\
u_t+uu_x &=\aT(\rho,u).
\end{split}\right.
\end{equation}
The forcing $\aT(\rho,u)$  takes the form $\aT(\rho,u)=[\aL_\phi,u](\rho):=\aL_\phi(\rho u)-\aL_\phi(\rho)u$,  which involves the density $\rho$, the velocity $u$, and a convolution kernel $\phi$,
\begin{equation}\label{eq:aL}
\aL_\phi(f):=\int_{\R} \phi(|x-y|) (f(y)-f(x))\dy.
\end{equation}
The system arises as  the macroscopic  description for large-crowd  dynamics of $N \gg1$ ``agents" driven by binary  interactions through velocity \emph{alignment},  \cite{CS2007},
\begin{equation}\label{eq:CS}
\left\{
\begin{split}
\dot{x}_i&=v_i,\\
\dot{v}_i& =\frac{1}{N}\sum_{j=1}^N \phi(|x_i-x_j|)(v_j-v_i),
 \end{split}\right.
\qquad (x_i,v_i)\in \Omega\times \R, \quad i=1,2, \ldots, N.
\end{equation}
The kernel  $\phi$ regulates the binary interactions among agents in $\Omega$. In the original setting of \cite{CS2007},  $\phi$ is assumed  positive, bounded influence function.  Many aspects of the formal passage from \eqref{eq:CS} to \eqref{eq:FH} are discussed in e.g., \cite{HT2008,CCP2017} and references therein; consult \cite{Pes2015, PS2017} for singular $\phi$'s. The important dynamical feature of the model is encoded in its long time behavior describing a \emph{flocking phenomenon}, in which the crowd of agents  congregates within a finite diameter $D(t) = \sup_{i,j}|x_i(t) - x_j(t)|  < D_\infty <\infty$, while aligning their velocities, 
$ \sup_{i,j}|v_i(t) - v_j(t)|  \stackrel{t\rightarrow \infty}{\longrightarrow} 0$, thus  
approaching the conserved  average velocity, 
$v_j(t) \stackrel{t\rightarrow \infty}{\longrightarrow} \frac{1}{N} \sum_k v_k(0)$.
 Starting with the seminal work of Cucker and Smale paper \cite{CS2007} and the follow-up works \cite{HT2008, HL2009, MT2014, TT2014, CCTT2016, ST2017b} and reference therein, it has become clear that in order to achieve \emph{unconditional flocking} in either the agent-based or the macroscopic descriptions \eqref{eq:CS},\eqref{eq:FH},  the system has to  involve  long range interactions so that $\int \phi(r) \dr = \infty$. The drawback of such an assumption in the context of Cucker-Smale model \eqref{eq:CS} is that each agent has to ``count'' \underline{all} its $(N-1)$ neighbors, close and far with equal footing. To remove this deficiency, Motsch and Tadmor introduced in \cite{MT2011} an adaptive averaging protocol in which each neighboring agents is counted by its relative influence. Thus, the normalization pre-factor $1/N$ on the right of \eqref{eq:CS} is replaced by $1/\sum_j \phi(|x_i-x_j|)$,  leading to  the Eulerian dynamics \eqref{eq:FH} with \emph{non-symmetric} forcing  $\aT(\rho,u)= [\aL_\phi,u](\rho)/(\phi*\rho)$. The model is argued in \cite{MT2011} as more realistic in both --- close to and away from  equilibrium regimes, but its lack of symmetry  is less amenable to the spectral analysis available in the symmetric Cucker-Smale model \eqref{eq:CS}. An alternative approach was proposed by us in \cite{ST2017a,ST2017b}, where nearby interactions are highlighted by the singularity of the interaction kernel at the origin, thus ``adapting''  different footing of neighboring agents by placing substantially smaller weights to those agents at far away distances relative to those nearby.
A natural example  is given by the power-law singularities $|x|^{-(1+\a)}, \ \a >0$. We consider the system \eqref{eq:FH} on the torus $\T$  with the $2\pi$-periodized version of such  kernels\footnote{We can in fact have an arbitrarily large period.}
\[
\phi_\a(x) := \sum_{k\in \Z} \frac{1}{|x+ 2\pi k|^{1+\a}}, \qquad 0<\a< 2.
\]
which  preserve the essential long range but less dominant interactions.  In this case the operator $\aL_\a=\aL_{\phi_\a}$ becomes the (negative of) classical \emph{fractional Laplacian}, $\aL_\a=-\L^\a$, which we denote
\[
\L^\a u(x) = \int_\R (u(x) - u(x+z)) \frac{\dz}{|z|^{1+\a}} = \int_\T (u(x) - u(x+z)) \phi_\a(z) \dz, \quad \L^\a=(-\Delta)^{\a/2}.
\]
Here and below we assume that $u(\cdot,t)_{|\T}$ and likewise, $\rho(\cdot,t)_{|\T}$, are   extended periodically onto the line $\R$. The commutator forcing on the right hand side of the momentum equation in \eqref{eq:FH} then becomes a fractional elliptic operator:
\begin{equation}\label{eq:Ta}  \tag*{(\theequation)$_{\a}$}
\aT(\rho,u)= - [\L^\a,u](\rho)(x) = \int_\R \rho(x+z) (u(x+z) - u(x)) \frac{\dz}{|z|^{1+\a}}, \qquad 0<\a <2,
\end{equation}
with the density controlling uniform ellipticity. Written in this form,  system \eqref{eq:FH} resembles the fractional Burgers equation with non-local non-homogeneous dissipation. 

In \cite{ST2017a} we proved global existence of smooth solutions of \eqref{eq:FH}, \ref{eq:Ta} in the range  $1\leq \a \leq 2$, with focus on the most difficult critical case $\a=1$. 
To this end we utilized refined tools from regularity theory of fractional parabolic equations
to verify a Beale-Kato-Majda (BKM) type continuation criterion which guarantees that the solution can be extended beyond $T$ provided 
$\displaystyle \int_0^T |u_x(\cdot, t)|_\infty \dt <\infty$. 
Building upon the technique developed in \cite{ST2017a}, in \cite{ST2017b} we proved that all regular solutions converge exponentially fast to a so called \emph{flocking state}, consisting of a traveling wave, $\overline{\rho}(x,t)=\rho_\infty(x-t\overline{u})$, with a fixed speed $\overline{u}$,
\begin{equation}\label{e:flock}
| u(\cdot,t)- \bar{u}|_X + |\rho(\cdot,t) - \bar{\rho}(\cdot,t)|_Y  \stackrel{t\rightarrow \infty}{\longrightarrow}0, \qquad  \overline{u}:=\frac{\cP_0}{\cM_0}.
\end{equation}
Here the average velocity, $\overline{u}$, is dictated by the conserved mass and momentum, 
\[
\cM_0=\int_\T \rho_0(x) \dx, \quad \cP_0=\int_\T (\rho_0u_0)(x) \dx.
\]
\ifx
. To recall the concept we introduce the set of flocking state solutions consisting of constant velocities, $\bar{u}$, and  traveling density waves
$\bar{\rho}=\rho_\infty(x-t\bar{u})$:
\begin{equation}\label{ }
\cF = \{ ( \bar{u}, \bar{\rho}):  \bar{u} \equiv \mbox{constant}, \bar{\rho}(x,t) = \rho_\infty(x-t \bar{u}) \}.
\end{equation}
We say that a solution $(u(\cdot,t),\rho(\cdot,t))$ converges to a flocking state  $(\bar{u},\bar{\rho}) \in \cF$ in space $X\times Y$ if

If the rate of decay in \eqref{e:flock} is exponential the process is accompanied by the classical \emph{fast alignment}, where $u \ra \bar{u}$ exponentially fast. In the present case of symmetric interactions, the conservation of averaged mass and momentum,
\[
\cM(t):=\frac{1}{2\pi}\int_{\T} \rho(x,t)dx \equiv \cM_0, \qquad
 \cP(t):=\frac{1}{2\pi} \int_{\T} (\rho u)(x,t)dx  \equiv \cP_0
 \]
implies that a limiting flocking velocity is given by $\bar{u}=\cP_0/\cM_0$.
\fi

Parallel to the release of \cite{ST2017a,ST2017b}, Do et.al. in \cite{DKRT2017} treated the case $0<\a<1$, where they proved global existence result with fast alignment of velocities. Although on the face of it, the system for that $\a$-range seems supercritical, one can employ the known conservation law for $e = u_x - \L^\a \rho$ to conclude a priori uniform $C^{1-\a}$ H\"older  regularity of the velocity, so that  the equation  \eqref{eq:FH}, \ref{eq:Ta} is kept critical in the range $0<\a<1$. In \cite{DKRT2017}, the authors use construction of a modulus of continuity, the celebrated method implemented in treating many critical equations such as Burgers and, most notably, critical SQG equation by Kiselev et. al. \cite{KNV2007}, in order to verify a Beale-Kato-Majda type  criterion $\displaystyle \int_0^T |\rho_x(\cdot,t)|^2_\infty \dt <\infty$, to guarantee continuation of the solution  beyond $T$.

In this present paper we revisit the parameter range  $0<\a<1$ using the fractional parabolic technique developed in our earlier works for the range $1\leq \a < 2$. As in \cite{ST2017b}, our methodology will be to extract \emph{quantitative enhancement} estimates for the dissipation term, using an adaptation of the non-linear maximum principle as in Constantin and Vicol's proof for the critical SQG, \cite{CV2012}, that  yields global existence and, moreover, allows us to  completely describe the long time behavior --- exponential convergence towards a flocking state. The main result  summarized in the following theorem covers the global regularity and flocking behavior for singular kernels in the unified range $0<\a<2$. The ($1\leq \a<2$)-part of the theorem was covered already in \cite[Theorem 1.3]{ST2017b}, quoted in \eqref{e:flock} with $(X,Y)=(H^3,H^{2+\a})$. The ($0< \a<1$)-part of the theorem is new.

\begin{theorem}[Flocking for singular kernels of fractional order $\a\in (0,2)$]\label{t:flock-singular}
Consider the system \eqref{eq:FH},\ref{eq:Ta} with singular kernel $\phi_\alpha(x) =|x|^{-(1+\a)}$, $0 < \a <2$, on the periodic torus $\T$, subject to  initial conditions $(\rho_0,u_0) \in  H^{2+\a} \times H^3  $ away from the vacuum. Then it admits a unique global solution  $(\rho,u)\in L^\infty([0,\infty);  H^{2+\a} \times H^3)$. Moreover, the solution converges exponentially fast to a flocking state $\bar{\rho}=\rho_\infty(x-t\bar{u})\in H^{2+\a}$ traveling with a finite speed $\bar{u}$, so that for any $s<2+\a$ there exists $C=C_s, \delta=\delta_s$ with
\begin{equation}\label{e:urxxx}
 |u(t)-\bar{u}|_{H^3} +|\rho(t) - \bar{\rho}(t) |_{H^{s}} \leq C e^{-\d t}, \qquad t>0, \qquad 
 \overline{u}:= \frac{\cP_0}{\cM_0}.
\end{equation}
\end{theorem}

We recall that the global existence part for $0<\a <1$ was first derived  in Do et. al. \cite{DKRT2017}. Our alternative proof  is along the lines of -- and in fact simpler to handle than,  the borderline case $\a=1$ in \cite{ST2017a}. The result is a consequence of  \lem{l:rslope} below, which gives a direct control on BKM continuation criteria $|\rho_x(\cdot,t)|_\infty$, and consequently on $|u_x(\cdot,t)|_\infty$, uniformly in time. Most of our work is then devoted  for obtaining quantitative bounds on  long time behavior of the slopes and higher order derivatives of the solution in the ($0<\a<1$)-part of the theorem.

\section{Preliminary a priori bounds}
We start by listing several structural features of the system \eqref{eq:FH},\ref{eq:Ta} and some preliminary a priori bounds of its solutions. We refer to \cite{DKRT2017,ST2017a,ST2017b} for details.

$\bullet$ (\emph{Control of higher order regularity}). The starting point is the conservation law for a new quantity :
\begin{equation}\label{e:e}
e_t + (ue)_x = 0, \qquad e:= u_x - \L^\a\rho.
\end{equation}
Paired with the mass equation we find that the ratio $ {e}/{\rho}$ satisfies the transport equation
\begin{equation}\label{ }
\DDt  ({e}/{\rho}) := (\p_t +u\p_x)  ({e}/{\rho}) = 0.
\end{equation}
Hence, starting from sufficiently smooth initial condition with $\rho_0$ away from vacuum, this gives a priori  pointwise bound
\begin{equation}\label{erho0}
|e(x,t)| \lesssim \rho(x,t).
\end{equation}
This argument can be bootstrapped to higher orders \cite[Sec. 2]{ST2017a}:  the next order quantity  $Q= (e/\rho)_x / \rho$ is transported
\begin{equation}\label{e:Q}
(\p_t +u\p_x) Q = 0, \qquad Q:= (e/\rho)_x / \rho
\end{equation}
hence solving for $e'(\cdot,t)$ we obtain the a priori pointwise bound
\begin{equation}\label{erho1}
|e'(x,t)| \lesssim |\rho'(x,t)| + \rho(x,t).
\end{equation}
This can be  iterated to any order yielding the high-order bounds
\begin{equation}\label{erhok}
|e^{(k)}(x,t)| \lesssim |\rho^{(k)}(x,t)| +\ldots + \rho(x,t), \qquad k=0.1.2., \ldots.
\end{equation}
 As observed in \cite{ST2017a}, the smallest order $L^2$-based regularity class for which \eqref{e:Q} can be understood classically, and hence  \eqref{erho0} holds at every point is the class $u\in H^3$, and \eqref{erho0} is the lowest order law among \eqref{erhok} which allows to close energy estimates. The corresponding regularity class for density $\rho$ follows from its connection to $u$ through the $e$-quantity which itself is of lower order. Hence, $\rho \in H^{2+\a}$. Indeed, it is proved in \cite{ST2017a} for $1\leq \a < 2$ and  in \cite{DKRT2017} for $0<\a<1$, that for any initial condition   $(\rho_0, u_0) \in H^{2+\a}\times  H^3 $ away from vacuum there exists a unique local solution in the same class $(\rho,u)\in L^\infty([0,T);  H^{2+\a}\times H^3 )$. We note that since the argument \cite{ST2017a} for $1\leq \a<2$  is not using the dissipative structure of the commutator term, it can be easily adapted to the case $0<\a<1$. Both results \cite{ST2017a} and \cite{DKRT2017} are accompanied by a BKM type continuation criterion which enables to extend the solution beyond any finite $T$.

$\bullet$ (\emph{Pointwise bound on the density}). We have the pointwise lower- and upper-bound on the density  globally on the interval of existence
\begin{equation}\label{e:densb}
0< c_0 \leq \rho(x,t) \leq C_0, \qquad  x\in\T, \ t \geq 0,
\end{equation}
where the constants $c_0$ and $C_0$ depend only on the initial condition.  This was established in \cite{ST2017b} following a weaker lower bound $\rho \gtrsim 1/(1+t)$ found in \cite{ST2017a,DKRT2017}.  

$\bullet$ (\emph{Strong alignment}). The variation of the velocity,  $\max_{y} u(y,t)-\min_{y} u(y,t)$, is contracting  exponentially fast,
\[
\ddt V(t)\le -   c_1 V(t), \qquad  V(t):=\max_{y}u(y,t) - \min_{y} u(y,t),
\]
hence there is an exponentially fast alignment of velocities to their average value   $u(x,t) \rightarrow \bar{u}=\cP_0/\cM_0$. 

$\bullet$ (\emph{Fractional parabolic enhancement}).  The parabolic nature of both the momentum and mass equations is an essential structural feature of the system that has been used in all of the preceding works. Using the $e$-quantity we can write
\begin{equation}
\rho_t+u\rho_x+e\rho= -\rho \L^\a \rho.
\end{equation}
The drift $u$ and the forcing $e \rho$ are bounded a priori due to the maximum principle  stated above. Moreover,  utilizing the boundedness of  $\rho$ and of $e=u_x-\L^\a\rho$ we immediately conclude for $0<\a < 1$  that $u(\cdot,t)\in C^{1-\a}$ uniformly in time. Hence, the mass  equation falls under the general class of fractional parabolic equations,
\[
w_t + b \cdot\n_x w = \cL_\a w + f, \qquad \cL_\a w(x)= \int_\R K(x,z,t) (w(x+z)-w(x))\dz
\]
with a diffusion operator associated with the  singular kernel
$K(x,z,t) = \rho(x+z) |z|^{-(1+\a)}$, and $f\in L^\infty$, $b\in C^{1-\a}$. Regularity of these equations has been the subject of active research in recent years. In particular, the result of  Silvestre \cite{S2012}, see also  Schwab and Silverstre \cite{SS2016}, states that  there exists a $\g>0$ such that for all $t>0$,
\begin{equation}\label{e:gamma}
| \rho|_{C^\g(\T \times [1,2))} \lesssim  |\r|_{L^\infty(0,2)} + |\rho e |_{L^\infty(0,2)}.
\end{equation}
Since the right hand side is uniformly bounded on the entire line we have obtained uniform bounds on $C^\g$-norm starting, by rescaling, from any positive time.

\section{Proof of the main result}

\subsection{Existence of global smooth solutions}
We begin with proving a uniform bound $|\rho_x(\cdot,t)|_\infty <\infty$.  In particular, we then have a uniform bound on $|\L^\a \rho|_\infty$, $e$, and hence on $|u'|_\infty$ and this readily implies global existence by the BKM criterion $\displaystyle \int_0^T |u_x(\cdot,t)|\dx <\infty$.
To simplify notations, we now use ${\{\cdot\}}^{'}, {\{\cdot\}}^{''}$ and so on to denote spatial differentiation.

\begin{lemma}\label{l:rslope} Under the assumptions stated of Theorem \ref{t:flock-singular} the following  uniform bound holds
\begin{equation}\label{ }
\sup_{t \ge 0 }  |\rho'(\cdot,t) |_\infty  <  \infty.
\end{equation}
\end{lemma}
\begin{proof}
Taking the derivative of the density equation we obtain
\[
\p_t \rho' + u \rho'' + u' \rho' + e'\rho + e \rho' = - \rho' \L^\a\rho - \rho \L^\a \rho',
\]
and expressing, $u' = e + \L ^\a \rho$, we rewrite the $\rho'$-equation as
\[
\p_t \rho' + u \rho'' + e'\rho + 2e \rho' = - 2\rho' \L^\a \rho - \rho \L^\a \rho'.
\]
Multiplying by $\rho'$ and evaluating the equation at the point $x_+$ which maximize 
$|\rho'(x_+,t)|=\max_x |\rho'(x,t)|$ we obtain
\begin{equation}\label{e:rho-dera}
\frac12 \p_t |\rho'_+|^2 + e'_+ \rho_+ \rho'_+ + 2 e_+ |\rho'_+|^2 = -2 |\rho'_+|^2 \L^\a \rho_+ - \rho_+ \rho'_+ \L^\a \rho'_+ =:-2 |\rho'|^2\cdot I + II.
\end{equation}
In view of \eqref{e:densb} and \eqref{erho1} the whole nonlinear term on the left hand side can be estimated by
\[
\big|e'_+ \rho_+ \rho'_+ + 2 e_+ |\rho'_+|^2\big| \leq c_2 |\rho'_+|^2.
\]
Next, in view of the lower-bound  $\rho \geq c_0$, we have
\begin{equation}\label{e:disslower}
II= \rho_+ \rho'_+ \L^\a \rho'_+ \geq \frac{1}{2} c_0 \dD_\a\rho'(x_+), 
\end{equation}
where
\[
\dD_\a\rho'(x) := \int_\R  \frac{|\rho'(x) - \rho'(x+z)|^2}{|z|^{1+\a}} \dz.
\]
By the nonlinear maximum principle of  \cite{CV2012}, at the maximal point $x=x_+$ we have 
\[
\dD_\a\rho'(x_+) \geq c_3 \frac{|\rho'_+|^{2+\a}}{|\rho|_\infty} \geq c_4 |\rho'_+|^{2+\a},
\]
and hence 
\begin{equation}\label{e:rho-derx}
II = - \rho_+ \rho'_+ \L^\a \rho'_+ \leq -c_{5}|\rho'_+|_\infty^{2+\a}, \qquad c_{5}=\frac{1}{2}c_0c_4.
\end{equation}
We now get back to estimating the term $I=\L^\a \rho$ in \eqref{e:rho-dera}.  The estimates are not restricted to the maximal point $x_+$ so we temporarily drop the subscript $\{\cdot\}_+$. Let $\psi\in C^\infty$ be the usual even cut-off function with $\psi(z) = 1$ for $|z|<1$ and $\psi(z) = 0$ for $|z|>2$. Denote $\psi_r(z) = \psi(z/r)$,  and decompose
\[
\begin{split}
\L^\a \rho(x)  & =  \int \psi_r(z) \frac{\rho(x) - \rho(x+z)}{|z|^{1+\a}} \dz + \int_{|z| <2\pi} (1-\psi_r(z)) \frac{\rho(x) - \rho(x+z)}{|z|^{1+\a}} \dz \\
&\ \ \ + \int_{2\pi <|z|} (1-\psi_r(z))  \frac{\rho(x) - \rho(x+z)}{|z|^{1+\a}} \dz =: I_1+I_2+I_3.
\end{split}
\]
The last integral, $I_3$, is bounded by a constant multiple of $|\rho|_\infty$, which is uniformly bounded, $\leq c_6$. In the intermediate integral we use $C^\g$-regularity of $\rho$ and the fact that the region of integration is restricted to $|z|>r$. So, we obtain
\[
I_2= \left| \int_{|z| <2\pi} (1-\psi_r(z)) \frac{\rho(x) - \rho(x+z)}{|z|^{1+\a}} \dz \right| \leq c_7 r^{\g-\a}.
\]
For the first small-scale integral, we use that $|z|^{-1-\a} = -\frac{1}{\a} \p_z( z |z|^{-1-\a})$ and integrate by parts to obtain
\[
I_1=  \int \psi_r(z) \frac{\rho(x) - \rho(x+z)}{|z|^{1+\a}} \dz = \frac{1}{\a} \int \psi'_r(z) \frac{\rho(x) - \rho(x+z)}{|z|^{1+\a}} z \dz -  \frac{1}{\a} \int \psi_r(z) \frac{\rho'(x+z)}{|z|^{1+\a}} z \dz .
\]
In the first integral we use $C^\g$ regularity to  obtain an upper-bound $\lesssim r^{\g-\a}$; as to the second, since $\psi_r$ is even we can add the term $\rho'(x)$ inside,
\[
 \frac{1}{\a} \int \psi_r(z) \frac{\rho'(x+z)}{|z|^{1+\a}} z dz = \frac{1}{\a} \int \psi_r(z) \frac{\rho'(x+z)-\rho'(x)}{|z|^{1+\a}} z \dz ,
\]
and using H\"older, the last integral does not exceed $\leq c_8 (\dD_\a \rho')^{1/2}(x) r^{1-\a/2}$. Putting all these estimates of $I_1,I_2$ and $I_3$ together, we obtain the bound for the nonlinear term
$-2|\rho'|^2 I$, 
\begin{align}
\begin{split}\label{e:Drhs}
\big| |\rho'_+|^2 \L^\a \rho_+ \big| & \lesssim c_6 |\rho'_+|^2 + c_7 |\rho'_+|^2 r^{\g - \a} + c_8  |\rho'_+|^2 (\dD_\a \rho'_+)^{1/2}(x) r^{1-\a/2} \\
& \leq c_6 |\rho'_+|^2 + c_7 |\rho'_+|^2 r^{\g - \a} + \frac{c_0}{4}  \dD_\a \rho'(x_+) + c_9 r^{2-\a} |\rho'|^4.
\end{split}
\end{align}
The third term on the right, $\displaystyle \frac{c_0}{4}  \dD_\a \rho'(x_+)$ is absorbed into \eqref{e:disslower}, leaving us with the dissipation of $\displaystyle \frac{1}{2}II \leq -\frac{c_5}{2} |\rho'_+|^{2+\a}$ in \eqref{e:rho-derx}. 
Setting $\displaystyle r = \frac{c_{10}}{ |\rho'_+|}$ with sufficiently small $c_{10}$, we see that the second and fourth terms on the right hand side of \eqref{e:Drhs} are absorbed into the dissipation term $\frac{1}{2}II$. With such choice of $r$, the final bound of \eqref{e:rho-dera}  reads,
\begin{equation}\label{e:rho-derb}
\p_t |\rho'_+|^2 \leq c_{11}|\rho'_+|^2 +c_{12} |\rho'_+|^{2+\a - \g} - c_{13} |\rho'_+|^{2+\a},
\end{equation}
which implies the claimed control of $|\rho'(\cdot,t)|_\infty$.
\end{proof}

\subsection{Main theorem --- step 1: exponential decay towards a flocking state}

 To establish the stated exponential decay of $|u_x(\cdot,t)|$ we first prepare with the following refinement of the nonlinear maximum principle, \cite{CV2012}
extending \cite[Lemma 3.3]{ST2017b}.

\begin{lemma}[Enhancement of dissipation by small amplitudes] \label{l:enhance} Let $u \in C^1(\T)$ be a given function with amplitude $V = \max u - \min u$. There is an absolute constant $c_1 >0$ such that the following pointwise estimate holds
\begin{equation}\label{eq:x1}
\dD_\a u'(x) = \int_\R  \frac{|u'(x) - u'(x+z)|^2}{|z|^{1+\a}} \dz \geq c_1 \frac{|u'(x)|^{2+\a}}{V^\a}, \qquad  V=\max u-\min u. 
\end{equation}
In addition, there is an absolute constant $c_2>0$ such that for all $B>0$ one has
\begin{equation}\label{eq:x2}
\dD_\a u'(x) \geq B |u'(x)|^2 - c_2 B^{\frac{1+\a}{\a}} V^2.
\end{equation}
\end{lemma}
\begin{proof} Let $\psi_r$ be as in the proof of \lem{l:rslope}. Discarding the positive term $|u(x+z)|^2$ we obtain
\[
\begin{split}
\dD_\a u'(x) & \geq \int_{|z| >r} (1- \psi_r(z)) \frac{ |u'(x)|^2 - 2 u'(x+z) u'(x) }{|z|^{1+\a}} \dz\\
& = c_1 |u'(x)|^2 r^{-\a} - 2 u'(x)\int_{|z|>r} (1- \psi_r(z)) \frac{u'(x+z)}{|z|^{1+\a}} \dz.
\end{split}
\]
Now, using  $u'(x+z)  \equiv (u(x+z) - u(x))_z$ we integrate by parts in the second integral to obtain
\[
\begin{split}
\lefteqn{\int_{|z|>r} (1- \psi_r(z)) \frac{u'(x+z)}{|z|^{1+\a}} \dz} \\
& \quad  = \int_{r<|z|<2r} \psi_r'(z) \frac{u(x+z) - u(x)}{|z|^{1+\a}} \dz +(1+\a) \int_{|z|>r} (1-\psi_r(z)) \frac{u(x+z) - u(x)}{|z|^{3+\a}} z\dz.
\end{split}
\]
Both integrals are bounded by a constant multiple of $V r^{-(1+\a)}$. Hence
\[
\dD_\a u'(x) \geq  c_1 |u'(x)|^2 r^{-\a}- c_2 |u'(x)| V r^{-(1+\a)}.
\]
Picking $\displaystyle r = \frac{2c_2 V}{c_1 |u'(x)|}$ we obtain \eqref{eq:x1}. Picking 
$\displaystyle r = B^{-(1/\a)}$ and using Young's inequality,
\[
\dD u'(x) \geq c_1 B |u'(x)|^2 - c_2 |u'(x)|V B^{\frac{1+\a}{\a}} \geq c_3 B |u'(x)|^2 - c_4 B^{\frac{2+\a}{\a}} V^2,
\]
we obtain  \eqref{eq:x2}.
\end{proof}

\begin{lemma} \label{l:vanslope} Under the assumptions of Theorem \ref{t:flock-singular} there exist constants $C, \d>0$ such that for all $t>0$ one has
\begin{equation}\label{ }
|u'(\cdot,t) |_\infty \leq Ce^{-\d t}.
\end{equation}
\end{lemma}
\begin{proof}
Differentiating the $u$-equation and evaluating at a point of maximum we find
\begin{equation}\label{eq:cubic}
\ddt |u'|^2 \leq |u'|^3 + \aT(\rho', u) u'  + \aT(\rho, u') u', \qquad \aT(\rho,u):=-\L^\a(\rho u)+u\L^\a(\rho).
\end{equation}
Pertaining to the dissipation term, let us observe
\[
(u'(y) - u'(x))u'(x)  = - \frac12 |u'(y) - u'(x)|^2 + \frac12( |u'(y)|^2 - |u'(x)|^2 )\le - \frac12 |u'(y) - u'(x)|^2.
\]
Thus, in view of density bounds \eqref{e:densb},
\[
\aT(\rho, u') u'  (x) \le - c_1 \dD_\a u'(x).
\]
The dissipation encoded in  $-c_1\dD_\a u'(x)$ cannot control the full cubic term $|u'|^3$ on the right of \eqref{eq:cubic}; yet as noted earlier, the term $|u'|$  is uniformly bounded (by the bounds of $|\L^\a\rho|_\infty$ and $|e|_\infty$) and in view of the enhancement \lem{l:enhance},
\[
|u'|^3 \lesssim |u'|^{2+\a} \lesssim V^\a(t) D_\a u', \qquad V(t)= \max_y u(y,t)-\min_y u(y,t).
\] 
Thus, the latter bound on $|u'|^3$  can be absorbed into dissipation term, at least  after a finite time at which $V(t)$ becomes small enough below certain threshold, $V(t) < c_1$. 

Let us turn to the remaining term $\aT(\rho', u) u' $.   We have
\[
\begin{split}
| \aT(\rho', u) u' | & = |u'| \int_{|z|<2\pi}| \rho'(x+z)| \frac{|u(x+z) -u(x)|}{|z|^{1+\a}} \dz \\
& \ \ \ + |u'| \int_{|z|>2\pi}| \rho'(x+z)| \frac{|u(x+z) -u(x)|}{|z|^{1+\a}} \dz\\
& \leq |u'|_\infty^2 |\rho'|_\infty + |u'|_\infty|\rho'|_\infty V \leq c_2|u'|_\infty^2 + E,
\end{split}
\]
where $E$ denotes a generic exponentially decaying quantity. In view of \eqref{eq:x2}, the quadratic term gets absorbed into dissipation leaving only exponentially decaying source term:
\[
\ddt |u'|^2 \leq E - c_3 |u'|^2,
\]
for all $t > t_0$ for some large $t_0$. The result follows by integration.
\end{proof}

We are now ready to prove existence of a flocking pair, at this stage in rough spaces.

\begin{lemma}\label{l:flock-bdd} Under the assumptions of \thm{t:flock-singular} there exist $C,\d >0$  and a flocking pair $(\bar{u}, \bar{\rho}) \in \cF$, $\bar{\rho} \in C^{1-\e}$, for every $\e>0$, such that
\begin{equation}\label{e:rhobar}
 |\rho(\cdot,t) - \bar{\rho}(\cdot,t) |_{\infty} \leq C e^{-\d t}, \qquad t>0.
\end{equation}
Thus, $\cF$ contains all limiting states of the system \eqref{eq:FH}.
\end{lemma}
\begin{proof} The proof is identical to one given in \cite{ST2017b}. We include it for completeness.
Clearly, the velocity goes to its natural limit $\bar{u} = \cP_0/ \cM_0$.  We pass to the moving reference frame and denote $\widetilde{\rho}(x,t) := \rho(x+ t \bar{u},t)$. We see that $\widetilde{\rho}$ satisfies
\[
\widetilde{\rho}_t + (u - \bar{u}) \widetilde{\rho}_x + u_x \widetilde{\rho} = 0,
\]
where all the $u$'s are evaluated at $x+ t \bar{u}$. According to the established bounds we have $\displaystyle | \widetilde{\rho}_t |_\infty < C e^{-\d t}$.
This proves that $\widetilde{\rho}(\cdot, t) $ is Cauchy as $t \ra \infty$, and hence there exists a unique limiting state, $\rho_\infty(x)$, such that
\[
| \widetilde{\rho} (\cdot,t) - \rho_\infty(\cdot)|_\infty < C_1 e^{-\d t}.
\]
Denoting $\bar{\rho}(\cdot,t)=\rho_\infty(x-t\bar{u})$
 completes the proof of \eqref{e:rhobar}. The membership of $\bar{\rho}$ in $C^{1-\e}$ follows from \lem{l:rslope} and the compactness.
\end{proof}

\subsection{Main theorem --- step 2: decay of higher derivatives}

We start by showing exponential decay of $|u''|_\infty$. As before we denote by $E= E(t)$ any quantity with an exponential decay. For example, at this point we know that $|u'|_\infty =E$ and $V = E$.  According to \lem{l:enhance} applied to $u''$, we have the following pointwise bounds
\begin{equation}\label{e:Duxx}
\begin{split}
\dD_\a u''(x) & \geq \frac{|u''(x)|^{2+\a}}{E},\\
\dD_\a u''(x) & \geq B |u''(x)|^2 - C(B) E.
\end{split}
\end{equation}
Due to these bounds the dissipation term absorbs all $(2+\a)$-power terms $C|u''|^{2+\a}$ as well as quadratic terms with bounded coefficients $C|u''|^2$, and by Young any linear term $E|u''|$ with exponentially decaying coefficient. The cost of this absorbing is a free source term of type $E$.

\begin{lemma} \label{l:vanslope2} Under the assumptions of Theorem \ref{t:flock-singular}, there are constants $C, \d>0$ such that for all $t>0$ one has
\begin{equation}\label{ }
|u''(\cdot, t) |_\infty \leq Ce^{-\d t}.
\end{equation}
\end{lemma}
\begin{proof} Evaluating the $u''$-equation at a point of maximum and performing the same initial steps as in \lem{l:vanslope} we obtain
\begin{equation}\label{e:uxx}
\ddt |u''|^2 \leq E |u''|^2 - c_0 \dD_\a u''(x) + \aT(\r'', u) u'' + 2 \aT(\rho',u')u''.
\end{equation}
As elaborated above, the quadratic term can be absorbed into dissipation by cost of an exponentially decaying source:
\[
\ddt |u''|^2 \leq E - c_1 \dD_\a u''(x) + \aT(\r'', u) u'' + 2 \aT(\rho',u')u''.
\]
We now focus on $\aT(\r'', u) u''$. Unfortunately, at this point we do not have any uniform control on $|\rho''|$. Thus, we will need to move one or $1-\a$ derivative from $\rho''$.  To achieve this we add and subtract $z u'
(x)$ inside the integral. We obtain
\begin{align*}
\aT(\r'', u) u'' = &  \ u''(x)u'(x) \int \rho''(x+z) \frac{z}{|z|^{1+\a}} \dz \\
 & + u''(x) \int  \rho''(x+z) (u(x+z) - u(x)-zu'(x)) \frac{\dz}{|z|^{1+\a}} \\
 & =:u''(x)u'(x)\cdot I+u''(x)\cdot II.
\end{align*}
We now integrate by parts  both integrals, $I$ and $II$.  In the first we obtain
\[
\begin{split}
I= \int \rho''(x+z) \frac{z}{|z|^{1+\a}} \dz &= \int (\rho'(x+z) - \rho'(x))_z \frac{z}{|z|^{1+\a}} \dz \\
 & = \a \int (\rho'(x+z) - \rho'(x)) \frac{\dz}{|z|^{1+\a}}  = - \a \L^\a \rho'(x).
\end{split}
\]
Note that $\L^\a \rho'(x) = e' - u''$, and $|e'|\lesssim |\rho'| < C$. Consequently,
\[
|u''(x)u'(x)\cdot I| = | - \a \L^\a \rho'(x) u''(x)u'(x) | \leq  E |u''|^2 + E|u''|,
\]
both are absorbed into dissipation with an extra $E$-term.  In the second integral, we obtain
\begin{equation}\label{e:aux3}
\begin{split}
II = & \ -   \int \rho'(x+z) (u'(x+z) - u'(x)) \frac{\dz}{|z|^{1+\a}} \\
&+ c \int  \rho'(x+z) (u(x+z) - u(x)-zu'(x)) \frac{z \dz}{|z|^{3+\a}}
\end{split}
\end{equation}
Splitting each integral into $|z|<2\pi$ and $|z|>2\pi$ regions, and using Taylor in the small scale regions we immediately obtain the bound
$\lesssim  |\rho'| |u'| + |\rho'| |u''| \leq E + c |u''|$.
The corresponding term $u''(x)\cdot II$ is therefore bounded by $E |u''| + c |u''|^2$, which is  again absorbed into dissipation. We conclude that the whole term $\aT(\r'', u) u''$ is dominated by dissipative term plus an $E$-source.

It remains to notice that the $\aT(\rho',u')u''$ term is precisely given by the first integral on the right hand side of \eqref{e:aux3}, which has been estimated already. We  arrive at
\[
\ddt |u''|^2 \leq E - c \dD u''(x) \lesssim E -  |u''(x) |^2.
\]
This finishes the proof.
\end{proof}

To proceed, let us note that we  have automatically obtained the uniform bound
\begin{equation}\label{ }
\sup_t  |\L^\a \rho'(\cdot, t)|_\infty <\infty.
\end{equation}

We are now in a position to perform final estimates in the top regularity class $H^3 \times H^{2+\a}$.

\begin{lemma} \label{l:vanslope3} Under the assumptions of Theorem \ref{t:flock-singular}, there are constants $C, \d>0$ such that for all $t>0$ one has
\begin{equation}\label{ }
\begin{split}
|u'''(\cdot,t) |_2 & \leq Ce^{-\d t} \\
|\L^\a \rho''(\cdot,t)|_2 & \leq C.
\end{split}
\end{equation}
\end{lemma}

\begin{proof}[Proof of \lem{l:vanslope3}]
Let us write the equation for $u'''$:
\begin{equation}\label{e:u'''}
u'''_t + u u'''_x + 4u'u'''+ 3u'' u'' = \aT(\r''',u)+ 3 \aT(\r'',u') + 3\aT(\r',u'') + \aT(\r,u''').
\end{equation}
Testing with $u'''$ we obtain (we suppress integral signs and note that $\int u''u''u''' = 0$)
\begin{equation}\label{e:eb}
\begin{split}
\ddt |u'''|_2^2 &= - 7 u'(u''')^2 + 2 \aT(\r''',u)u''' +6 \aT(\r'',u')u''' + 6 \aT(\r',u''))u''' + 2 \aT(\r, u''')u''' \\
& \leq E |u'''|_2^2 - c_0  \int \dD_\a u''' \dx+2 \aT(\r''',u)u''' +6 \aT(\r'',u')u''' + 6 \aT(\r',u'')u'''.
\end{split}
\end{equation}
Note that $ \int \dD_\a u''' \dx = |u'''|_{H^{\a/2}}^2$.   From \lem{l:enhance} we have the lower bound
\begin{equation}\label{e:DuxxxInt}
\int_\T \dD u''' \dx  \geq B |u'''|_2^2 - C(B) E, \text{ for any } B>0.
\end{equation}
Again, the dissipation absorbs all quadratic terms and linear terms with $E$-coefficient. Manipulations below are much similar to the ones we performed in the proof of \lem{l:vanslope2} So, we proceed straight with computations. We have
\[
\begin{split}
 \aT(\r''',u)u''' & = \int_{\T \times \R} \frac{\rho'''(x+z)(u (x+z)-u(x))  u'''(x)}{|z|^{1+\a}} \dz \dx\\
 &= \int_{\T \times \R} \frac{\rho'''(x+z) z u'(x)  u'''(x)}{|z|^{1+\a}} \dz \dx \\
 &\ \ \ + \int_{\T \times \R} \frac{\rho'''(x+z)(u (x+z)-u(x) - z u'(x) )  u'''(x)}{|z|^{1+\a}} \dz \dx \\
& = \a \int_{\T}   \L^\a \rho''(x) u'(x) u'''(x) \dx +  \int_{\T \times \R} \frac{\rho''(x+z)(u' (x+z)- u'(x) )  u'''(x)}{|z|^{1+\a}} \dz \dx  \\
& \ \ \ +  \int_{\T \times \R} \frac{\rho''(x+z)(u (x+z)-u(x) - z u'(x) )  u'''(x) z }{|z|^{3+\a}} \dz \dx  \\
& \leq  |u'|_\infty | \L^\a \rho''|_2 |u'''|_2 + |\rho''|_2 |u'''|_2 |u'|_\infty + |\rho''|_2 |u'''|_2 |u''|_\infty  \leq E   | \L^\a \rho''|_2 |u'''|_2 \\
&\leq E |e''|_2^2 + E|u'''|_2^2,
\end{split}
\]
where the last term in absorbed into dissipation. Next,
\[
\aT(\r'',u')u''' = \int_{\T \times \R} \frac{\rho''(x+z)(u' (x+z)- u'(x) )  u'''(x)}{|z|^{1+\a}} \dz \dx ,
\]
which is precisely an integral we already estimated below. Finally,
\[
\begin{split}
\aT(\r',u'')u''' & = \int_{\T \times \R} \frac{\rho'(x+z)(u'' (x+z)- u''(x) )  u'''(x)}{|z|^{1+\a}} \dz \dx \\
& =  \int_{\T } \int_{|z|<2\pi} \frac{\rho'(x+z)(u'' (x+z)- u''(x) )  u'''(x)}{|z|^{1+\a}} \dz \dx \\
&+  \int_{\T } \int_{|z|>2\pi} \frac{\rho'(x+z)(u'' (x+z)- u''(x) )  u'''(x)}{|z|^{1+\a}} \dz \dx \\
&\leq |\rho'|_\infty |u'''|_2^2 + |\rho'|_\infty |u''|_\infty |u'''|_2 \leq c |u'''|_2^2 + E |u'''|_2.
\end{split}
\]
This term is entirely absorbed into dissipation.  We have obtained the estimate
\begin{equation}\label{e:uxxx2}
\ddt |u'''|_2^2 \leq E + E   |e''|^2_2  - \int \dD_\a u''' \dx .
\end{equation}
It remains to close with a bound on $e''$:
\begin{equation}\label{e:exx}
\begin{split}
\ddt |e''|_2^2 & \leq  3 u' e''e'' + 2 u'' e' e'' + u''' e e'' \leq E |e''|_2^2 +  E|e''|_2 + |u'''|^2|e''|^2 \\
&\leq E |e''|_2^2 + \frac{E}{\d} + \d |e''|_2^2 + \frac{1}{\d} |u'''|_2^2 + \d |e''|^2_2,
\end{split}
\end{equation}
for every $\d>0$. Combining with \eqref{e:uxxx2} and absorbing all quadratic terms of $|u'''|_2$  we obtain for $X =  |u'''|_2^2 + |e''|_2^2 $:
\begin{equation}\label{ }
\ddt X \leq \frac{E}{\d} + c \d X.
\end{equation}
This implies that the exponential growth rate of $X$ is at most $c\d$, which can be made arbitrarily small. In particular this implies that $|e''|_2$ has arbitrarily small exponential rate. Going back to \eqref{e:uxxx2} we find that $E   |e''|^2_2$ is another $E$-type term, since the $E$ coefficient has a finite negative decay rate. Consequently, we obtain
\begin{equation}\label{e:uxxx3}
\ddt |u'''|_2^2 \leq E - c |u'''|_2^2,
\end{equation}
which proves the result for $|u'''|_2$.  To finish the bound on density we go back to the $e''$-equation with the obtained exponential decay of $u'''$:
\begin{equation}\label{e:exx}
\ddt |e''|_2^2 \leq  3 u' e''e'' + 2 u'' e' e'' + u''' e e'' \leq E(|e''|_2^2 + |e''|_2).
\end{equation}
This readily implies global uniform bound on $|e''|_2$, and hence on $|\rho'''|_2$.
This proves the lemma.
\end{proof}

As a consequence we readily obtain the full statement of \thm{t:flock-singular}. In particular, the convergence for densities stated in \eqref{e:urxxx} follows by interpolation between exponential decay in $L^\infty$ and uniform boundedness in $H^{2+\a}$. The fact that $\bar{\rho} \in H^{2+\a}$ is a simple consequence of uniform boundedness of $\rho(t)$ in $H^{2+\a}$ and weak compactness.


\begin{thebibliography}{10}

\bibitem{CCP2017}
J.~Carrillo, Y.-P. Choi, and S.~Perez.
\newblock A review on attractive-repulsive hydrodynamics for consensus in
  collective behavior.
\newblock In N.~Bellomo, P.~Degond, and E.~Tadmor, editors, {\em Active
  Particles. Advances in Theory, Models, and Applications}, volume~1.
  Birkh\"auser, 2017.

\bibitem{CCTT2016}
J.~A. Carrillo, Y.-P. Choi, E. Tadmor, and C. Tan.
\newblock Critical thresholds in 1{D} {E}uler equations with non-local forces.
\newblock {\em Math. Models Methods Appl. Sci.}, 26(1):185--206, 2016.

\bibitem{CV2012}
P. Constantin and V. Vicol.
\newblock Nonlinear maximum principles for dissipative linear nonlocal
  operators and applications.
\newblock {\em Geom. Funct. Anal.}, 22(5):1289--1321, 2012.

\bibitem{CS2007}
F. Cucker and S. Smale.
\newblock Emergent behavior in flocks.
\newblock {\em IEEE Trans. Automat. Control}, 52(5):852--862, 2007.

\bibitem{DKRT2017}
T.~Do, A.~Kiselev, L.~Ryzhik, and C.~Tan.
\newblock Global regularity for the fractional euler alignment system.
\newblock 2017.
\newblock https://arxiv.org/abs/1701.05155.


\bibitem{HL2009}
S.~Y Ha and J.~G Liu,
\newblock A simple proof of the {Cucker-Smale} flocking dynamics and mean-field
  limit.
\newblock {\em Communications in Mathematical Sciences}, 7(2), (2009) 297-325.

\bibitem{HT2008}
S.-Y. Ha and E. Tadmor.
\newblock From particle to kinetic and hydrodynamic descriptions of flocking.
\newblock {\em Kinet. Relat. Models}, 1(3):415--435, 2008.

\bibitem{KNV2007}
A.~Kiselev, F.~Nazarov, and A.~Volberg.
\newblock Global well-posedness for the critical 2{D} dissipative
  quasi-geostrophic equation.
\newblock {\em Invent. Math.}, 167(3):445--453, 2007.

\bibitem{MT2011}
S. Motsch and E. Tadmor.
\newblock A new model for self-organized dynamics and its flocking behavior.
\newblock {\em J. Stat. Phys.}, 144(5):923--947, 2011.

\bibitem{MT2014}
S. Motsch and E. Tadmor.
\newblock Heterophilious dynamics enhances consensus.
\newblock {\em SIAM Rev.}, 56(4):577--621, 2014.


\bibitem{Pes2015}
J. Peszek. 
\newblock Discrete Cucker-Smale flocking model with a weakly singular weight. 
\newblock {\em SIAM J. Math. Anal.}, 47(5), 3671-3686.

\bibitem{PS2017}
D. Poyato and J. Soler.
\newblock Euler-type equations and commutators in singular and hyperbolic   limits of kinetic Cucker-Smale models, 
\newblock {\em Math. Models Methods Appl. Sci.} 27, 1089 (2017) 1089-1152.

\bibitem{SS2016}
R.~W. Schwab and L. Silvestre.
\newblock Regularity for parabolic integro-differential equations with very
  irregular kernels.
\newblock {\em Anal. PDE}, 9(3):727--772, 2016.

\bibitem{ST2017a}
R.~Shvydkoy and E.~Tadmor.
\newblock Eulerian dynamics with a commutator forcing.
\newblock {\em Trans. Mathematics and Applications} 1(1) (2017) 1-26.

\bibitem{ST2017b}
R.~Shvydkoy and E.~Tadmor.
\newblock Eulerian dynamics with a commutator forcing ii: flocking.
\newblock {\em Discrete and Continuous Dynamical Systems-A}, in press.

\bibitem{S2012}
L. Silvestre.
\newblock H\"older estimates for advection fractional-diffusion equations.
\newblock {\em Ann. Sc. Norm. Super. Pisa Cl. Sci. (5)}, 11(4):843--855, 2012.

\bibitem{TT2014}
E. Tadmor and C.Tan.
\newblock Critical thresholds in flocking hydrodynamics with non-local
  alignment.
\newblock {\em Philos. Trans. R. Soc. Lond. Ser. A Math. Phys. Eng. Sci.},
  372(2028):20130401, 22, 2014.

\end{thebibliography}

\end{document}